\pdfoutput=1
\documentclass{article}
\usepackage[utf8]{inputenc}
\usepackage{amsmath}
\usepackage{amssymb}
\usepackage{amsthm}
\usepackage{tikz-cd}
\usepackage{mathtools}
\usepackage{hyperref}
\usepackage{comment}
\hypersetup{colorlinks=true}
\usepackage[totalwidth=480pt, totalheight=680pt]{geometry}

\title{The Diagonal Dimension of Curves}
\author{Noah Olander}
\date{}

\newtheorem{thm}{Theorem}

\newtheorem{prop}{Proposition}
\newtheorem{lemma}{Lemma}
\newtheorem{corollary}{Corollary}
\theoremstyle{definition}

\newtheorem*{conjecture}{Conjecture}

\begin{document}

\maketitle 

\begin{abstract}
    We prove Conjecture 4.16 of the paper \cite{ELAGIN2021334} of Elagin and Lunts; namely, that a smooth projective curve of genus at least $1$ over a field has diagonal dimension $2$. 
\end{abstract}

\section{Introduction}

Let $k$ be a field. All unadorned products will be over $k$. Recall that on $\mathbf{P}^n_k \times \mathbf{P}^n_k$ there is the famous Beilinson resolution of the diagonal
$$
0 \to \mathcal{O}_{\mathbf{P}^n_k} (-n) \boxtimes \Omega ^n_{\mathbf{P}^n_k}(n) \to \cdots \to \mathcal{O}_{\mathbf{P}^n_k} (-1) \boxtimes \Omega ^1 _{\mathbf{P}^n_k}(1) \to \mathcal{O}_{\mathbf{P}^n_k} \boxtimes \mathcal{O}_{\mathbf{P}^n_k} \to \mathcal{O}_\Delta \to 0.
$$
The existence of this resolution implies that every object of the derived category of $\mathbf{P}^n_k$ can be built from the object $G = \mathcal{O}_{\mathbf{P}^n_k} \oplus \mathcal{O}_{\mathbf{P}^n_k} (-1) \oplus \cdots \oplus \mathcal{O}_{\mathbf{P}^n_k} (-n)$ using direct sums, shifts, direct summands, and \emph{at most} $n$ cones. More precisely, $D_{perf} (\mathbf{P}^n_k) = \langle G \rangle _{n+1}$. The reader can see \cite{rouquier_2008} for the definition of $\langle G \rangle _{n+1}$. 

Now let $X$ be a smooth, separated scheme of finite type over $k$. The example of $\mathbf{P}^n_k$ suggests two notions of dimension for $X$: the \emph{Rouquier dimension} of $X$, denoted $\mathrm{Rdim}(X)$, is the least integer $n$ such that there exists an object $G \in D_{perf} (X)$ with $D_{perf} (X) = \langle G \rangle _{n+1}$. The \emph{diagonal dimension} of $X$ \emph{over} $k$, denoted $\mathrm{Ddim}(X/k)$ or more often just $\mathrm{Ddim}(X)$, is the least integer $n$ such that there exist $G, H \in D_{perf} (X)$ such that $\mathcal{O}_\Delta \in \langle G \boxtimes H \rangle _{n+1} \subset D_{perf} (X \times_k X)$. These definitions were originally given in \cite{rouquier_2008} and \cite{10.1093/imrn/rnr124}, respectively, and make sense more generally. See \cite{ELAGIN2021334}. Some by now classical arguments show that
\begin{equation}
\label{equation-inequalities}
\mathrm{dim} (X) \leq \mathrm{Rdim}(X) \leq \mathrm{Ddim}(X) \leq 2 \mathrm{dim} (X),
\end{equation}
see \cite{ELAGIN2021334}. 

The existence of the Beilinson resolution for $\mathbf{P}^n_k$ shows that $\mathrm{Rdim}(\mathbf{P}^n_k) = \mathrm{Ddim} (\mathbf{P}^n_k) = n$.  In \cite{Orl09}, Orlov shows that if $X$ is a smooth curve then $\mathrm{Rdim}(X) = 1$, and he conjectures that for any smooth, quasi-projective scheme $X$ over $k$, $\mathrm{Rdim}(X) = \mathrm{dim} (X)$. If $X$ is a smooth projective curve over $k$ with $H^1(X, \mathcal{O}_X) = 0$, then $\mathrm{Ddim}(X) = 1$. We have seen this if $X = \mathbf{P}^1_k$, and more generally it is Proposition \ref{prop-converse}. At this point, there seem to be no examples in the literature of smooth varieties $X$ with $\mathrm{Ddim} (X) > \mathrm{dim} (X)$, but in \cite{ELAGIN2021334}, Elagin and Lunts write that they expect this inequality to hold for most smooth projective varieties. They make the following conjecture:

\begin{conjecture}[Elagin and Lunts]
Let $X$ be a smooth projective curve over $k$ such that $H^1 (X, \mathcal{O}_X) \neq 0$. Then 
$$
\mathrm{Ddim}(X) = 2.
$$
\end{conjecture}

In this paper, we prove the Conjecture of Elagin and Lunts, see Theorem \ref{theorem-main}. By the inequalities (\ref{equation-inequalities}), to prove the Conjecture, we need only show $\mathrm{Ddim}(X) > 1$. The following lemma (which the reader who does not want to learn the precise definition of $\langle \underline{\hspace{.5 em}} \rangle _{n}$ should feel free to take as a black box) explains what this means: 

\begin{lemma}
\label{lemma-translation}
Let $X$ be a smooth curve over $k$. Then $\mathrm{Ddim} (X) = 1$ if and only if there exist perfect complexes $E, F, G, H$ on $X$ and a morphism $E \boxtimes F \to G \boxtimes H$ in $D_{perf} (X \times X)$ such that $\mathcal{O}_\Delta$ is a direct summand of its cone.
\end{lemma}

\begin{proof}
If there exist such $E, F, G, H$ and $E \boxtimes F \to G \boxtimes H$ then $\mathcal{O}_\Delta \in \langle (E \oplus G) \boxtimes (F \oplus H)\rangle_2 $, hence $1 \leq \mathrm{Ddim}(X) \leq 1$. 

Conversely, if $\mathrm{Ddim}(X) = 1$, choose perfect complexes $K, L$ on $X$ such that $\mathcal{O}_\Delta \in \langle K \boxtimes L \rangle _2$. This means that $\mathcal{O}_\Delta$ is a direct summand of the cone of a morphism
$$
\varphi : \bigoplus_{i \in I finite} K \boxtimes L [m_i] \to \bigoplus_{j \in J finite} K \boxtimes L [n_j].
$$
Then there is the following trick: $\bigoplus _i E_i \boxtimes F_i $ is a direct summand of $ (\bigoplus _i E_i ) \boxtimes (\bigoplus_j F_j)$, hence we may find perfect complexes $E, F, G, H$ on $X$ and $A, B$ on $X \times X$ such that 
\begin{align*}
    E \boxtimes F &= \big(\bigoplus_{i \in I } K \boxtimes L [m_i]\big) \oplus A \\
    G \boxtimes H &= \big(\bigoplus_{j \in J } K \boxtimes L [n_j]\big) \oplus B.
\end{align*}
Then there is the morphism
$$
\begin{pmatrix}
\varphi & 0 \\
0 & 0
\end{pmatrix}: E \boxtimes F \to G \boxtimes H .
$$
The cone of this morphism contains $\mathrm{Cone} (\varphi)$ as a direct summand, which in turn contains $\mathcal{O}_\Delta$ as a direct summand, so we are done.
\end{proof}

Thus we only have to prove that there do not exist such $E, F, G, H$ and $E \boxtimes F \to G \boxtimes H$. A nice exercise for the reader is to prove this under the additional assumption that $E, F, G, H$ are semistable vector bundles on $X$. The strategy of our proof is to reduce to this case. In Section \ref{section-reduction} we reduce to the case that $E, F, G, H$ are vector bundles, and then we complete the proof in Section \ref{section-bundle-case} by considering their Harder--Narasimhan filtrations. The way to make these reductions is to recognize that a perfect complex on $X$ has a lot of structure: It is a direct sum of shifts of coherent sheaves, each coherent sheaf is a sum of a vector bundle and a torsion module, and each vector bundle possesses a Harder--Narasimhan filtration. This leads to lots of structure on the box product of two perfect complexes on $X$. We are able to leverage this using Reduction Strategies 1 and 2 (see Section \ref{section-reduction}), which are easy strategies for taking a direct summand $S$ of the cone of a morphism $A \to B$ in a triangulated category and then showing that $S$ is actually a direct summand of the cone of a simpler morphism. 

For people more comfortable with Abelian categories than triangulated categories, we give the following interesting corollary of Theorem \ref{theorem-main}:

\begin{corollary}
\label{corollary-ext}
Let $X$ be a smooth projective curve over $k$ such that $H^1(X, \mathcal{O}_X) \neq 0$. Then for any resolution of $\mathcal{O}_\Delta \in Coh(X \times X)$ of the form
$$
 \cdots \to \mathcal{E}_n \boxtimes \mathcal{F}_n \to \cdots \to \mathcal{E}_1 \boxtimes \mathcal{F}_1 \xrightarrow{d_1} \mathcal{E}_0 \boxtimes \mathcal{F}_0 \to \mathcal{O}_\Delta \to 0 
$$ 
with $\mathcal{E}_i, \mathcal{F}_i$ vector bundles on $X$, the associated extension class $\xi \in \mathrm{Ext}^2 _{\mathcal{O}_{X \times X}} (\mathcal{O}_\Delta , \mathrm{Ker} (d_1) )$ is nonzero. 
\end{corollary}

We find Corollary \ref{corollary-ext} interesting because it depends so strongly on the form of the resolution. Notice how absolutely wrong it becomes if we used instead the resolution $0 \to \mathcal{O}(- \Delta ) \to \mathcal{O} \to \mathcal{O}_\Delta \to 0$. 

\begin{proof}[Proof of Corollary]
This is because $\xi$ is the obstruction to the existence of a section of the map $\mathrm{Cone} (d_1) \to \tau_{\geq 0} (\mathrm{Cone} (d_1)) = \mathcal{O}_\Delta$, so if $\xi$ vanishes then $\mathcal{O}_\Delta$ is a direct summand of the cone of $d_1$.
\end{proof}

The author would like to thank Johan de Jong for helpful conversations and Alex Perry for pointing him to the paper \cite{ELAGIN2021334} and its interesting conjectures.

\section{Reduction to the Vector Bundle Case}
\label{section-reduction}

In this section, $X$ will be a smooth projective curve over $k$. Our goal is to prove Proposition \ref{prop-reduction}. We begin by studying the structure of the box product $E \boxtimes F$ with $E, F$ perfect complexes on $X$ and setting up some notation regarding this. 

\begin{lemma}
\label{lemma-unambiguous}
Let $\mathcal{E}, \mathcal{F}$ be coherent sheaves on $X$. Then 
\begin{equation*}
    \mathrm{Tor}_i^{\mathcal{O}_{X \times X}} (pr_1^* (\mathcal{E}), pr_2^*(\mathcal{F})) = 0
\end{equation*}
for $i>0$.
\end{lemma}

\begin{proof}
By the formulae for the stalk of pullbacks and Tor modules at a point, this immediately reduces to the following algebra problem: $A, B$ are DVRs with uniformizers $\pi_A, \pi_B$; $R$ is a regular local ring of dimension $2$ with local ring homomorphisms $A \to R$ and $B\to R$ such that $(\pi_A, \pi_B)$ is a regular system of parameters for $R$; and $M, N$ are finite modules over $A, B$. Then 
$$
\mathrm{Tor}_i^R (M \otimes _A R, N \otimes _B R) = 0
$$
for $i > 0$. 

To prove this, use the structure theorem for finite modules over a DVR. The free parts cannot contribute to the Tor modules, hence we may assume $M = A / (\pi_A ^m)$ and $N = B / (\pi_B ^n)$ for $m , n > 0$. Then the Tor groups in question are
$$
\mathrm{Tor}_i^R (R / (\pi_A ^m), R / (\pi_B ^n)) ,
$$
which vanish for $i > 0$ since $(\pi_A ^m, \pi_B ^n)$ is a regular sequence on $R$. 
\end{proof}

Thus if $E, F$ are perfect complexes on $X$, then $E \boxtimes F$ is a decomposable object of $D(X \times X)$ (i.e., isomorphic to the direct sum of its cohomology sheaves shifted appropriately; recall that every perfect complex on $X$ is decomposable) and we have the formula
\begin{equation*}
    H^k(E \boxtimes F) = \bigoplus _{i + j = k} H^i(E) \boxtimes H^j(F)
\end{equation*}
with the box products on the right hand side underived. Furthermore, since a coherent sheaf on $X$ is a direct sum of a vector bundle and a torsion sheaf, we see that $H^k(E \boxtimes F)$ is a direct sum of four parts which respectively have the form:
$$ \bigoplus bundle \boxtimes bundle, \bigoplus bundle \boxtimes torsion, \bigoplus torsion \boxtimes bundle, \bigoplus torsion \boxtimes torsion.
$$
We will sometimes refer to the first part as $H^k(E \boxtimes F)_{free},$ the last part as $H^k(E \boxtimes F)_0$ (for $0$-dimensional support), and the sum of the last three parts as $H^k(E \boxtimes F)_{tors}$.

The following two lemmas are simple calculations we will need in the proof of Proposition \ref{prop-reduction}. 

\begin{lemma}
\label{lemma-ext-calculations}
Let $E, F$ be perfect complexes on $X$. Then
$$
R\mathrm{Hom} _{\mathcal{O}_{X \times X}}(\mathcal{O}_\Delta , E \boxtimes F) = R \Gamma (X , E \otimes ^{\mathbf{L}}_{\mathcal{O}_X} F \otimes _{\mathcal{O}_X} ^{\mathbf{L}} T_X)[-1],
$$
and 
$$ R\mathrm{Hom}_{\mathcal{O}_{X \times X}} (E \boxtimes F, \mathcal{O}_\Delta) = R \mathrm{Hom}_{\mathcal{O}_X} (E \otimes _{\mathcal{O}_X}^{\mathbf{L}} F , \mathcal{O}_X).
$$
\end{lemma}

\begin{proof}
The second formula is just adjunction between $R \Delta _*$ and $L\Delta ^*$. We prove the first. Using the standard resolution of $\mathcal{O}_\Delta$ by $\mathcal{O}$ and $\mathcal{O}(-\Delta),$ we obtain a distinguished triangle
\begin{equation*}
    R\mathcal{H}om_{\mathcal{O}_{X \times X}}(\mathcal{O}_\Delta , E \boxtimes F) \to E \boxtimes F \to E \boxtimes F (\Delta) \to .
\end{equation*}
Thus, 
$$
 R\mathcal{H}om_{\mathcal{O}_{X \times X}}(\mathcal{O}_\Delta , E \boxtimes F) \cong (E \boxtimes F)\otimes ^{\mathbf{L}} _{\mathcal{O}_{X \times X}} \mathcal{O}_\Delta (\Delta)[-1] \cong R\Delta _* (E \otimes ^{\mathbf{L}}_{\mathcal{O}_X} F \otimes _{\mathcal{O}_X} ^{\mathbf{L}} T_X)[-1],
 $$
and so
$$
R\mathrm{Hom} _{\mathcal{O}_{X \times X}}(\mathcal{O}_\Delta , E \boxtimes F) \cong R \Gamma (X \times X, R\Delta _* (E \otimes ^{\mathbf{L}}_{\mathcal{O}_X} F \otimes _{\mathcal{O}_X} ^{\mathbf{L}} T_X))[-1] \cong R \Gamma (X , E \otimes ^{\mathbf{L}}_{\mathcal{O}_X} F \otimes _{\mathcal{O}_X} ^{\mathbf{L}} T_X)[-1].
$$
\end{proof}

\begin{lemma}
\label{lemma-vanishing} 
Let $\mathcal{E}, \mathcal{F}$ be torsion coherent sheaves on $X$ and $\mathcal{A}$ a vector bundle on $X \times X$. Then 
$$
\mathrm{Ext}^i_{\mathcal{O}_{X \times X}} (\mathcal{E}\boxtimes \mathcal{F}, \mathcal{A}) = 0
$$
for $i = 0, 1$.
\end{lemma}

\begin{proof}
It suffices to show $\mathcal{E}xt^i_{\mathcal{O}_{X \times X}} (\mathcal{E}\boxtimes \mathcal{F}, \mathcal{A}) = 0$ for $i = 0, 1$ by the relation between local and global Ext's. This reduces as in Lemma \ref{lemma-unambiguous} to a local algebra problem: $A, B$ are DVRs with uniformizers $\pi_A, \pi_B$; $R$ is a regular local ring of dimension $2$ with local ring homomorphisms $A \to R$ and $B\to R$ such that $(\pi_A, \pi_B)$ is a regular system of parameters for $R$; $M, N$ are finitely generated torsion modules over $A, B$; and $F$ is a finite free $R$-module. Then 
$$
\mathrm{Ext}^i_R ((M \otimes _A R)\otimes _R  (N \otimes _B R), F) = 0
$$
for $i = 0,1$. 

By the structure theorem for finite modules over a DVR, we are allowed to assume $M = A / (\pi_A^m), N = B / (\pi_B^n)$ for $m, n>0$. Then the Ext groups in question are 
$$
\mathrm{Ext}^i_R (R / (\pi_A ^m, \pi_B ^n), F)
$$
which vanish for $i = 0 , 1$ since $(\pi_A^m, \pi_B^n)$ is a regular sequence on the free module $F$. 
\end{proof}

We have two strategies for taking a summand of a cone of a morphism and showing that it is actually a summand of a cone of a simpler morphism (which is exactly what we are trying to do -- see Proposition \ref{prop-reduction}).

\vspace{1em}

\noindent \textbf{Reduction Strategy 1.}
Here $\mathcal{T}$ is a triangulated category. Suppose we have a morphism $A \to B$ in $\mathcal{T}$ with cone $C$ and that $S$ is a direct summand of $C$, so that there are morphisms $i : S \to C$ and $p : C \to S$ with $p\circ i = \text{id}_S$. Suppose that there are distinguished triangles
\begin{align*}
    &A' \to A \to A'' \to A'[1] \\
    &B' \to B \to B'' \to B'[1]
\end{align*}
in $\mathcal{T}$. 

\begin{lemma}
Notation as above. If $\mathrm{Hom}(B', S) = 0$ and $\mathrm{Hom} (S, A''[1])=0$, then actually $S$ is a direct summand of the cone of $A' \to B''$. 
\end{lemma}

\begin{proof}
Write $D$ for the cone of $A \to B''$. We will first show that $S$ is a direct summand of $D$. By the octahedral axiom, there is a distinguished triangle
$$
B' \to C \to D \to B'[1]
$$
where the first arrow is the composition $B' \to B \to C$. What we have to show is that the projection $p : C \to S$ factors through $D$. This follows via a diagram chase from the assumption that $\mathrm{Hom} (B', S) = 0$. 

Thus $S$ is a direct summand of the cone of $A \to B''$ and a dual argument to the one in the first paragraph allows us to replace $A$ with $A'$. 
\end{proof}

\noindent \textbf{Reduction Strategy 2.} This is a slightly more specialized situation. Again $\mathcal{T}$ is a triangulated category and this time we have a morphism $A \oplus B \to C \oplus D$ with matrix
$$
\begin{pmatrix}
F & 0 \\
0 & 0
\end{pmatrix},
$$
and we have a direct summand $S$ of its cone. The cone is isomorphic to
$$ 
\mathrm{Cone} (F) \oplus D \oplus B[1].
$$

\begin{lemma}
Notation as above. Assume $\mathrm{Hom}(S, D) = 0$ and $\mathrm{Hom}(B[1], S) = 0$. Then $S$ is a direct summand of $\mathrm{Cone}(F)$.
\end{lemma}

\begin{proof} 
Say the inclusion $S \to \mathrm{Cone} (F) \oplus D \oplus B[1]$ is given by the maps $(a, b, c)$ and the retraction is given by the maps $(d, e, f).$ Then we have $da + eb + cf = 1$ and $b = 0 = f$ by assumption, hence $da = 1$.
\end{proof}

\begin{prop}
\label{prop-reduction}
Suppose there are perfect complexes $E, F, G, H$ on $X$ and a morphism $E \boxtimes F \to G \boxtimes H$ such that $\mathcal{O}_\Delta$ is a direct summand of its cone. Then there are vector bundles $\mathcal{E}, \mathcal{F}, \mathcal{G}, \mathcal{H}$ on $X$ and a morphism $\mathcal{E} \boxtimes \mathcal{F} \to \mathcal{G} \boxtimes \mathcal{H}$ such that $\mathcal{O}_\Delta$ is a direct summand of its cone. 
\end{prop}

\begin{proof} \underline{Step 1.} $\mathcal{O}_\Delta$ is a direct summand of the cone of $\tau_{\leq 1} (E \boxtimes F) \to \tau_{\geq 0} (G \boxtimes H)$.

\vspace{.5em}

This follows from the first reduction strategy since there are no maps $\tau_{<0}(G \boxtimes H) \to \mathcal{O}_\Delta$ and no maps $\mathcal{O}_\Delta \to \tau_{>1} (E \boxtimes F)[1]$ for degree reasons.

\vspace{.5em}

\noindent \underline{Step 2.} $\mathcal{O}_\Delta$ is a direct summand of the cone of a map 
$$H^0(E \boxtimes F) \oplus H^1(E \boxtimes F)[-1] \to H^0(G \boxtimes H) \oplus H^1(G \boxtimes H)[-1].$$

 Here we apply the second reduction strategy using the decompositions
 \begin{align*}
     \tau_{\leq 1} (E \boxtimes F) &= (H^0(E \boxtimes F) \oplus H^1(E \boxtimes F)[-1]) \oplus \tau_{ <0} (E \boxtimes F)\\
     \tau_{\geq 0}(G \boxtimes H) &= (H^0(G \boxtimes H) \oplus H^1(G \boxtimes H)[-1]) \oplus \tau_{>1} (G \boxtimes H).
 \end{align*}
For degree reasons, the morphism between them has matrix of the form
$$
\begin{pmatrix}
* & 0 \\
0 & 0
\end{pmatrix},
$$
and also 
$$
\mathrm{Hom} (\tau_{ <0} (E \boxtimes F)[1], \mathcal{O}_\Delta) = 0 = \mathrm{Hom}(\mathrm{O}_\Delta , \tau_{>1} (G \boxtimes H)), 
$$
so that reduction strategy gives the result.

\vspace{.5 em}

\noindent \underline{Step 3.} $\mathcal{O}_\Delta$ is a direct summand of the cone of a map 
$$H^0(E \boxtimes F) \oplus H^1(E \boxtimes F)_{0}[-1] \to H^0(G \boxtimes H)_{free} \oplus H^1(G \boxtimes H)[-1].$$

First reduction strategy. This applies because 
$$\mathrm{Hom} (H^0(G \boxtimes H)_{tors} , \mathcal{O}_\Delta) = 0 = \mathrm{Hom}(\mathcal{O}_\Delta , H^1(E \boxtimes F) / H^1(E \boxtimes F)_0 ).
$$
To prove the first equality we have to show that if $\mathcal{G}, \mathcal{H}$ are coherent sheaves on $X$ with at least one of them a torsion module, then $\mathrm{Hom} (\mathcal{G} \boxtimes \mathcal{H}, \mathcal{O}_\Delta) = 0.$ But this Hom is identified with $\mathrm{Hom}_{\mathcal{O}_X} (\mathcal{G} \otimes \mathcal{H} , \mathcal{O}_X)$ which vanishes since there are no maps from a torsion sheaf to a locally free sheaf. For the second equality we have to show $\mathrm{Hom} (\mathcal{O}_\Delta , \mathcal{E} \boxtimes \mathcal{F}) =0$ if $\mathcal{E}, \mathcal{F}$ are coherent sheaves on $X$ with at least one of them locally free. By Lemma \ref{lemma-ext-calculations} this Hom is identified with $H^{-1} (X, \mathcal{E} \otimes ^{\mathbf{L}} \mathcal{F} \otimes T_X) = 0$ since the tensor product is underived. 

\vspace{.5 em} 

\noindent \underline{Step 4.} $\mathcal{O}_\Delta$ is a direct summand of the cone of a map 
$$H^0(E \boxtimes F)_{free} \oplus H^1(E \boxtimes F)_{0}[-1] \to H^0(G \boxtimes H)_{free} \oplus H^1(G \boxtimes H)_0[-1].$$

Second reduction strategy using the decompositions 
\begin{align*}
    H^0(E \boxtimes F) \oplus H^1(E \boxtimes F)_{0}[-1] &= (H^0(E \boxtimes F)_{free} \oplus H^1(E \boxtimes F)_{0}[-1]) \oplus H^0(E \boxtimes F)_{tors} \\
    H^0(G \boxtimes H)_{free} \oplus H^1(G \boxtimes H)[-1] &= (H^0(G \boxtimes H)_{free} \oplus H^1(G \boxtimes H)_0[-1]) \oplus H^1(G \boxtimes H) / H^1(G \boxtimes H)_0 [-1].
\end{align*}
The morphism between them has matrix
$$
\begin{pmatrix}
* & 0 \\
0 & 0
\end{pmatrix}.
$$
Most of the vanishing required to prove this is easy and left to the reader, but we will explain why 
$$\mathrm{Hom}( H^1(E \boxtimes F)_0 , H^1(G \boxtimes H) / H^1(G \boxtimes H)_0) = 0.$$
It suffices to show that if $\mathcal{E}, \mathcal{F}, \mathcal{G}, \mathcal{H}$ are coherent sheaves on $X$ with $\mathcal{E}, \mathcal{F}$ torsion and at least one of $\mathcal{G}, \mathcal{H}$ a vector bundle, then $\mathrm{Hom} (\mathcal{E} \boxtimes \mathcal{F}, \mathcal{G} \boxtimes \mathcal{H}) = 0$. By the K\"unneth formula, we have
$$
\mathrm{Hom} (\mathcal{E} \boxtimes \mathcal{F}, \mathcal{G} \boxtimes \mathcal{H}) = \mathrm{Hom} (\mathcal{E}, \mathcal{G}) \otimes _k \mathrm{Hom} (\mathcal{F}, \mathcal{H}) = 0
$$
since one of $\mathcal{G}, \mathcal{H}$ is a vector bundle and there are no maps from a torsion sheaf to a vector bundle. 

Finally, we have
$$
\mathrm{Hom} (H^0(E \boxtimes F)_{tors}[1], \mathcal{O}_\Delta) = 0 = \mathrm{Hom} (\mathcal{O}_\Delta , H^1(G \boxtimes H) / H^1(G \boxtimes H)_0 [-1]) 
$$
for degree reasons, so that the second reduction strategy applies.

\vspace{.5 em}

\noindent \underline{Step 5.} $\mathcal{O}_\Delta$ is a direct summand of the cone of a map
$$
\varphi: H^0(E \boxtimes F)_{free} \to H^0(G \boxtimes H)_{free}.
$$

This is very similar to the second reduction strategy. We note that the map
$$H^0(E \boxtimes F)_{free} \oplus H^1(E \boxtimes F)_{0}[-1] \to H^0(G \boxtimes H)_{free} \oplus H^1(G \boxtimes H)_0[-1]$$
has diagonal matrix 
$$
\begin{pmatrix}
\varphi & 0 \\
0 & \psi
\end{pmatrix}
$$
with respect to the direct sum decompositions: $\mathrm{Hom} (H^0(E \boxtimes F)_{free}, H^1(G \boxtimes H)_0[-1]) = 0$ for degree reasons and $\mathrm{Hom}(H^1(E \boxtimes F)_{0}[-1], H^0(G \boxtimes H)_{free}) =0$ by Lemma \ref{lemma-vanishing}. Therefore, $\mathcal{O}_\Delta$ is a direct summand of $\mathrm{Cone} (\varphi) \oplus \mathrm{Cone}(\psi)$. Let the inclusion be given by maps $(a, b)$ and the retraction by maps $(c, d)$ so that $ca + db = \text{id}_{\mathcal{O}_\Delta}$. We are going to show that $db$ is zero on cohomology sheaves, and then we will be done since $ca$ will be a quasi-isomorphism. But $\mathcal{O}_\Delta$ has only the one nonzero cohomology sheaf so actually we will be done if we can show 
$$
H^0(d) : H^0(\mathrm{Cone} (\psi)) \to \mathcal{O}_\Delta 
$$
is zero. This follows since $H^0(\mathrm{Cone} (\psi))$ has zero-dimensional support, being a submodule of $H^1(E \boxtimes F)_0$.

\vspace{.5 em}

\noindent \underline{Step 6.} Conclusion.

\vspace{.5 em}
The Proposition now follows from the same trick as Lemma \ref{lemma-translation}. Since $H^0(E \boxtimes F)_{free}$ and $H^0(G \boxtimes H)_{free}$ are direct sums of box products of vector bundles, there are vector bundles $\mathcal{E}, \mathcal{F}, \mathcal{G}, \mathcal{H}$ on $X$ and  $\mathcal{A}, \mathcal{B}$ on $X \times X$ such that 
\begin{align*}
    \mathcal{E} \boxtimes \mathcal{F} &= H^0(E \boxtimes F)_{free} \oplus \mathcal{A} \\
    \mathcal{G} \boxtimes \mathcal{H} &= H^0(G \boxtimes H)_{free} \oplus \mathcal{B}.
\end{align*}
Take the morphism 
$$
\begin{pmatrix}
\varphi & 0 \\
0 & 0
\end{pmatrix}
$$
between them. Then $\mathcal{O}_\Delta$ is a direct summand of the cone of $\varphi$ which is a direct summand of the cone of this morphism, and we are done.
\end{proof}

\section{The Vector Bundle Case}
\label{section-bundle-case}

In this section we prove:

\begin{thm}
\label{theorem-main}
Assume $X$ is a smooth projective curve over $k$ with $H^1(X, \mathcal{O}_X) \neq 0$. Then the diagonal dimension of $X$ is 2. 
\end{thm}

We begin by reducing to the case in which $X$ is geometrically connected over $k$, i.e., $H^0(X, \mathcal{O}_X)=k$. 

\begin{lemma}
\label{lemma-connected}
Let $X$ be a smooth proper variety over $k$. Set $k' = H^0(X, \mathcal{O}_X)$, a finite separable field extension of $k$. Then $\mathrm{Ddim}(X/k) = \mathrm{Ddim}(X/k')$. 
\end{lemma}

\begin{proof}
The pullback diagram 
$$
\begin{tikzcd}
X \times _{k'} X \ar[r, "f"] \ar[d] & X \times _k X \ar[d] \\
\mathrm{Spec} (k') \ar[r, "\Delta_{k'/k}"] &\mathrm{Spec} (k') \times _k \mathrm{Spec} (k')
\end{tikzcd}
$$
shows that $f$ is both an open and closed immersion, since the same is true for $\Delta_{k'/k}$. Since $X \times _{k'} X$ is connected, $f$ is the inclusion of a connected component of $X \times _k X$. Since $X$ is connected and $f \circ \Delta _{X / k'} = \Delta _{X/k}$, we can characterize $f$ as the inclusion of the unique connected component of $X \times _k X$ containing the diagonal.

With this in mind, we have
$$
\mathcal{O}_{\Delta_{X/k}} \in \langle G \boxtimes _k H \rangle_{d+1} \iff \mathcal{O}_{\Delta_{X/k'}} = Lf^* (\mathcal{O}_{\Delta_{X/k}}) \in \langle Lf^* (G \boxtimes _k H) \rangle _{d+1} = \langle G \boxtimes _{k'} H \rangle _{d+1},
$$
completing the proof.
\end{proof}

Let $X$ be a smooth, projective, geometrically connected curve $k$. A vector bundle $\mathcal{E}$ on $X$ has a Harder--Narasimhan filtration:
\begin{equation*}
   0 = \mathcal{E}_0 \subset \mathcal{E}_1 \subset \cdots \subset \mathcal{E}_N = \mathcal{E}
\end{equation*}
and we will write $\mu_i = \mu_i (\mathcal{E}) = \mu (\mathcal{E}_{i}  / \mathcal{E}_{i-1})$. The $\mu_i$ are weakly decreasing. It will sometimes be convenient to index instead by slope. For this we will write $\mathcal{E}_{i} = \mathcal{E}^{\mu_i}$. We will also write $\mathcal{E}_{i} / \mathcal{E}_{i-1} = \text{gr}^{\mu_i} (\mathcal{E})$. Given a second vector bundle $\mathcal{F}$, we get a filtration of the box-product $\mathcal{E}\boxtimes \mathcal{F}$ with
\begin{equation*}
    \text{Fil}^{\gamma} = \sum_{\alpha + \beta \geq \gamma} \mathcal{E}^{\alpha} \boxtimes \mathcal{F}^{\beta}
\end{equation*}
The graded pieces of the filtration are
\begin{equation*}
    \text{gr}^{\gamma} = \bigoplus _{\alpha + \beta = \gamma} \text{gr}^{\alpha}(\mathcal{E}) \boxtimes \text{gr}^{\beta} (\mathcal{F}).
\end{equation*}
In particular, this is a filtration of $\mathcal{E} \boxtimes \mathcal{F}$ by sub-\emph{bundles}.

\begin{proof}[Proof of Theorem \ref{theorem-main}]
By Lemma \ref{lemma-connected}, we may assume $X$ is geometrically connected over $k$ of genus
$g \geq 1$. By Lemma \ref{lemma-translation} and Proposition \ref{prop-reduction}, it suffices to show $\mathcal{O}_\Delta$ is not a direct summand of the cone of a map $\mathcal{E} \boxtimes \mathcal{F} \to \mathcal{G} \boxtimes \mathcal{H}$ for any vector bundles $\mathcal{E}, \mathcal{F}, \mathcal{G}, \mathcal{H}$ on $X$. Suppose it is. Apply Reduction Strategy 1 using the short exact sequences 
 $$
0 \to \sum_{\alpha + \beta \geq 2g-2 }\mathcal{E}^\alpha \boxtimes \mathcal{F}^\beta \to \mathcal{E} \boxtimes \mathcal{F} \to \mathcal{Q} \to 0$$
$$
 0 \to \sum _{\alpha + \beta > 0} \mathcal{G} ^{\alpha} \boxtimes \mathcal{H}^\beta \to \mathcal{G} \boxtimes \mathcal{H} \to \mathcal{Q'} \to 0.$$

For this to work we need 
\begin{equation*}
    \mathrm{Hom} (\sum_{\alpha + \beta > 0} \mathcal{G} ^{\alpha} \boxtimes \mathcal{H}^\beta, \mathcal{O}_\Delta) = 0 = 
    \mathrm{Hom} (\mathcal{O}_\Delta , \mathcal{Q}[1]).
\end{equation*}
The first equality is true because $\sum_{\alpha + \beta > 0} \mathcal{G} ^{\alpha} \boxtimes \mathcal{H}^\beta$ has a filtration whose graded pieces are direct sums of box products $\mathcal{G}' \boxtimes \mathcal{H}'$ with $\mathcal{G}'$ and $\mathcal{H}'$ semistable bundles with positive sum of slopes, and no such box product has a morphism to $\mathcal{O}_\Delta$:
$$
\mathrm{Hom}(\mathcal{G}' \boxtimes \mathcal{H}', \mathcal{O}_\Delta) = \mathrm{Hom}(\mathcal{G}' \otimes \mathcal{H}', \mathcal{O}_X) = \mathrm{Hom}(\mathcal{G}', \mathcal{H}'^\vee) = 0
$$
since the source has higher slope than the target. The second equality is true because $\mathcal{Q}$ has a filtration whose graded pieces are direct sums of box products $\mathcal{E}' \boxtimes \mathcal{F}'$ with $\mathcal{E}', \mathcal{F}'$ semistable bundles with sum of slopes $> 2g-2$, and no such box product receives a map from $\mathcal{O}_\Delta [-1]$:
$$
\mathrm{Hom} (\mathcal{O}_\Delta , \mathcal{E}' \boxtimes \mathcal{F}'[1]) = H^0(X, \mathcal{E}' \otimes \mathcal{F}' \otimes T_X) = 0
$$
since $\mu (\mathcal{E}') + \mu (\mathcal{F}') + \mathrm{deg} ( T_X ) < 0$.

Therefore, $\mathcal{O}_\Delta$ is a direct summand of the cone of a morphism
$$
\varphi: \sum_{\alpha + \beta \geq 2g-2 }\mathcal{E}^\alpha \boxtimes \mathcal{F}^\beta \to \mathcal{Q}'.
$$
We now split into two cases:

\noindent \underline{Case 1.} $g \geq 2$.

In this case $\varphi = 0$. This is because the source has a filtration whose graded pieces are direct sums of objects of the form $\mathcal{E}' \boxtimes \mathcal{F}'$ with $\mathcal{E}', \mathcal{F}'$ semistable with sum of slopes $\geq 2g-2 > 0$ and the target has a filtration whose graded pieces are direct sums of objects of the form $\mathcal{G}' \boxtimes \mathcal{H}'$ with $\mathcal{G}', \mathcal{H}'$ semistable with sum of slopes $\leq 0$. Furthermore we have by the K\"unneth formula
\begin{equation*}
    \text{Hom} (\mathcal{E}' \boxtimes \mathcal{F}', \mathcal{G}' \boxtimes \mathcal{H}') = \text{Hom} (\mathcal{E}', \mathcal{G}') \otimes _k
    \text{Hom} (\mathcal{F}', \mathcal{H}')
\end{equation*}
which can only be non-zero if $\mu(\mathcal{E}') \leq \mu(\mathcal{G}')$ and $\mu(\mathcal{F}') \leq \mu(\mathcal{H}')$. But this contradicts the fact that $\mu(\mathcal{E}') + \mu (\mathcal{F}') > \mu (\mathcal{G}') + \mu (\mathcal{H}')$.

Hence $\mathcal{O}_\Delta$ is a direct summand of $H^0(\mathrm{Cone} (\varphi)) = \mathcal{Q}'$ which is a vector bundle, a contradiction.

\noindent \underline{Case 2.} $g = 1$.

In this case we claim that $H^0(\mathrm{Cone} (\varphi))$ is still a vector bundle, leading to the same contradiction. This time the argument of Case 1 does not show that $\varphi$ is zero but it does show that we obtain a factorization
$$
\varphi: \sum_{\alpha + \beta \geq 2g-2 = 0 }\mathcal{E}^\alpha \boxtimes \mathcal{F}^\beta \twoheadrightarrow \text{gr}^0(\mathcal{E} \boxtimes \mathcal{F}) \to \text{gr}^0(\mathcal{G} \boxtimes \mathcal{H}) \subset \mathcal{Q}',
$$
and it suffices to show that the cokernel of 
$$
\psi: \text{gr}^0(\mathcal{E} \boxtimes \mathcal{F}) \to \text{gr}^0(\mathcal{G} \boxtimes \mathcal{H})
$$
is a vector bundle. The source is a direct sum of box products $\mathcal{E}' \boxtimes \mathcal{F}'$ with $\mathcal{E}', \mathcal{F}'$ semistable bundles with sum of slopes $= 0$, and similarly for the target. By the argument of Case 1 with the K\"unneth formula, the only nonzero components of $\psi$ are the maps $\mathcal{E}' \boxtimes \mathcal{F}' \to \mathcal{G}' \boxtimes \mathcal{H}'$ with $\mu (\mathcal{E}') = \mu (\mathcal{G}')$ and $\mu (\mathcal{F}') = \mu (\mathcal{H}'),$ i.e., $\psi$ is diagonal with respect to the given direct sum decomposition. Therefore, it suffices to show that the cokernel of each nonzero component $\mathcal{E}' \boxtimes \mathcal{F}' \to \mathcal{G}' \boxtimes \mathcal{H}'$ is a vector bundle. This follows from Lemma \ref{lemma-constant-rank} below.
\end{proof}

\begin{lemma}
\label{lemma-constant-rank}
Let $\mathcal{E}, \mathcal{F}, \mathcal{G}, \mathcal{H}$ be semistable bundles on a smooth projective geometrically connected curve $X$ over $k$ with $\mu(\mathcal{E}) = \mu (\mathcal{G})$ and $\mu(\mathcal{F}) = \mu (\mathcal{H})$. Then the kernel and cokernel of any map $\varphi: \mathcal{E} \boxtimes \mathcal{F} \to \mathcal{G} \boxtimes \mathcal{H}$ are vector bundles. 
\end{lemma}

\begin{proof}
It suffices to prove this after pullback along the flat covering $X_{\bar{k}} \to X$. Since the pullback of a semistable bundle on $X$ to $X_{\bar{k}}$ remains semistable (see \cite[Corollary 1.3.8]{31757520100101}), we may assume $k = \bar{k}$.

We must show that the function taking a closed point $p \in X \times X$ to  $\mathrm{rk}(\varphi \otimes k (p))$ is constant. To prove this, it is enough to show that the rank stays constant on each horizontal and vertical closed fiber $X \times \{y\}, \{x\} \times X$, $x, y$ closed points of $X$. The restriction of $\varphi$ to $X \times \{y\} \cong X$ is a map 
$$
\mathcal{E} \otimes_{\mathcal{O}_X} \mathcal{O}_X^{\oplus m} \to \mathcal{G} \otimes_{\mathcal{O}_X} \mathcal{O}_X^{\oplus n},
$$
where $m = \mathrm{rk} (\mathcal{F})$ and $n = \mathrm{rk} (\mathcal{H})$. This is a map between semistable bundles of the same slope, hence it has constant rank. The same argument works for the vertical fibers and we are done.
\end{proof}

We finish the paper by proving a converse to Theorem \ref{theorem-main}:

\begin{prop}
\label{prop-converse}
Let $X$ be a smooth projective curve over $k$ such that $H^1(X, \mathcal{O}_X) = 0$. Then $\mathrm{Ddim}(X) = 1$.
\end{prop}

\begin{proof}
%
By Lemma \ref{lemma-connected}, we may assume $X$ is geometrically connected over $k$. Then $X_{\bar{k}} \cong \mathbf{P}^1_{\bar{k}}$, i.e., $X$ is a Severi--Brauer variety of dimension $1$ over $k$. There exists a vector bundle $\mathcal{E}$ on $X$ such that $\mathcal{E}_{\bar{k}} \cong \mathcal{O}_{\mathbf{P}^1_{\bar{k}}}(-1) \oplus \mathcal{O}_{\mathbf{P}^1_{\bar{k}}}(-1)$ (take the dual of the bundle constructed in \cite[Baby Example 4]{kollar2016severibrauer}). Now consider the ideal sheaf $\mathcal{O}(-\Delta)$ of $\Delta \subset X \times X$. We claim that $\mathcal{O}(-\Delta)$ is a direct summand of $\mathcal{E} \boxtimes \mathcal{E}$. This will complete the proof: The cone of the composition $\mathcal{E} \boxtimes \mathcal{E} \to \mathcal{O}(-\Delta) \to \mathcal{O}_{X \times X} = \mathcal{O}_X \boxtimes \mathcal{O}_X$ will contain $\mathcal{O}_\Delta$ as a direct summand. 

Note that the base change of $\mathcal{O}(-\Delta)$ to $\bar{k}$ is $\mathcal{O}_{\mathbf{P}^1_{\bar{k}}}(-1)\boxtimes \mathcal{O}_{\mathbf{P}^1_{\bar{k}}}(-1)$ and so $\mathcal{E}_{\bar{k}} \boxtimes \mathcal{E}_{\bar{k}} \cong \mathcal{O}(-\Delta)_{\bar{k}}^{\oplus 4}$. But now \cite[Lemma 8]{kollar2016severibrauer} implies that if $\mathcal{F}$ is a vector bundle on $X \times X$ with $\mathcal{F}_{\overline{k}} \cong \mathcal{O}(-\Delta)_{\overline{k}} ^{\oplus d}$, then actually $\mathcal{F} \cong \mathcal{O}(-\Delta)^{\oplus d},$ and we are done.  
\end{proof}

\bibliographystyle{alpha}
\bibliography{references}

\textsc{Columbia University Department of Mathematics, 2990 Broadway, New York, NY 10027}

\href{mailto:nolander@math.columbia.edu}{nolander@math.columbia.edu}
\end{document}